\documentclass[reqno, 11pt]{amsart}
\usepackage{amssymb}%, natbib}
\usepackage[nesting]{hyperref}
\usepackage[pdftex]{graphicx}
\usepackage{listings}
\usepackage{multirow}
\usepackage{placeins}
\usepackage{color}
\usepackage{subfigure}
\usepackage{lscape}
\usepackage{dsfont}

\newcommand{\BlackBoxes}{\global\overfullrule5pt}

\BlackBoxes

\newtheorem{theorem}{Theorem}

\theoremstyle{definition}
\newtheorem{example}[theorem]{Example}
\newtheorem{remark}[theorem]{Remark}

\numberwithin{equation}{section} \numberwithin{theorem}{section}

\def\0{\kern0pt\-\nobreak\hskip0pt\relax}

\makeatletter
\AtBeginDocument{%
 \def\@serieslogo{%
 \vbox to\headheight{%
 \parindent\z@ \fontsize{6}{7\p@}\selectfont
 \vss}}}

\def\makeoverbar#1#2#3#4#5#6#7{%
 \setbox0=\hbox{$\m@th#2\mkern#5mu{{}#3{}}\mkern#6mu$}%
 \setbox1=\null \dimen@=#4\fontdimen8#13 \dimen@=3.5\dimen@
 \advance\dimen@ by \ht0 \dimen@=-#7\dimen@ \advance\dimen@ by \wd0
 \ht1=\ht0 \dp1=\dp0 \wd1=\dimen@
 \dimen@=\fontdimen8#13 \fontdimen8#13=#4\fontdimen8#13
 \rlap{\hbox to \wd0{$\m@th\hss#2{\overline{\box1}}\mkern#5mu$}}
 \fontdimen8#13=\dimen@}

\def\mylabel#1#2{{\def\@currentlabel{#2}\label{#1}}}

\makeatother

\def\RR{\mathbb{R}}
\def\PP{\mathbb{P}}
\def\NN{\mathbb{N}}

\begin{document}

%\listoftodos

\makeatletter \providecommand\@dotsep{5} \makeatother
%\listoftodos[Changes in Orange/Red To Do List in Green / Blue]\relax

\title[Partially Observable Risk-Sensitive Stopping Problems in Discrete Time]{Partially Observable Risk-Sensitive Stopping Problems in Discrete Time}

\author[N. \smash{B\"auerle}]{Nicole B\"auerle${}^*$}
\address[N. B\"auerle]{Department of Mathematics,
Karlsruhe Institute of Technology, D-76128 Karlsruhe, Germany}

\email{\href{mailto:nicole.baeuerle@kit.edu}
{nicole.baeuerle@kit.edu}}

\author[U. \smash{Rieder}]{Ulrich Rieder${}^\dag$}
\address[U. Rieder]{University of Ulm, D-89069 Ulm, Germany}

\email{\href{mailto:ulrich.rieder@uni-ulm.de} {ulrich.rieder@uni-ulm.de}}

%\thanks{${}^*$ Department of Mathematics, Karlsruhe Institute of Technology, D-76128 Karlsruhe, Germany }

\begin{abstract}
In this paper we consider stopping problems with partial observation under a general risk-sensitive optimization criterion for problems with finite and infinite time horizon. Our aim is to maximize the certainty equivalent of the stopping reward. We develop a general theory and discuss the Bayesian risk-sensitive  house selling problem as a special example. In particular we are able to study the influence of the attitude towards risk of the decision maker on the optimal stopping rule.
\end{abstract}
\maketitle

\vspace{0.5cm}
\begin{minipage}{14cm}
{\small
\begin{description}
\item[\rm \textsc{ Key words}]
{\small optimal stopping,  certainty equivalent, exponential utility, updating operator, value iteration}
\item[\rm \textsc{AMS subject classifications}]{\small Primary 62L15,  Secondary 90C40}
\end{description}
}
\end{minipage}

\section{Introduction}\label{BRsec:intro}
In this paper we consider stopping problems with partial observation under a general risk-sensitive optimization criterion for problems with finite and infinite time horizon. More precisely our aim is to maximize the certainty equivalent of the stopping reward over the time horizon. In case of an infinite time horizon we assume that we have strictly negative cost when we do not stop. The certainty equivalent of a random variable is defined by $U^{-1}(\mathbb{E} U(X))$ where $U$ is an increasing concave function. If $U(x)=x$ we obtain as a special case the classical risk-neutral decision maker. The case $U(x)=\frac1\gamma e^{\gamma x }$ is often referred to as 'risk-sensitive', however the risk-sensitivity is here only expressed in a special way through the risk-sensitivity parameter $\gamma\neq 0$.  More general, the certainty equivalent may be written (assuming enough regularity of $U$) as
\begin{equation}\label{BReq:Urep} U^{-1}\Big(\mathbb{E}\big[U(X)\big]\Big)\approx \mathbb{E} X - \frac12 l_U(\mathbb{E} X) Var[X]\end{equation}
where \begin{equation}\label{BReq:AP}  l_U(x) = -\frac{U''(x)}{U'(x)}\end{equation} is the {\em Arrow-Pratt} function of absolute risk aversion. In case of an exponential utility, this absolute risk aversion is constant (for a discussion see \cite{bpli03}).

In addition we suppose that the process is partially observable. More precisely we suppose that we have a jointly Markovian process $(X_n,Y_n)$ where only the first component is observable. Also the stopping reward depends only on the first component. This class of models includes in particular Bayesian models where the unobservable process reduces to an unknown parameter. Important applications are sequential probability ratio tests. A particular example can e.g. be found in \cite{t84}. The theory of risk-neutral stopping problems can e.g. be found in \cite{crs71} or \cite{shi08}.

Stopping problems with general utility functions are rarely treated in the literature. \cite{m00} considers the classical house selling problem with general utility but with complete observation. In a separate section we show  that most of the results in \cite{m00} also extend to the Bayesian case. \cite{KMY96} considers stopping problems with denumerable state space and arbitrary utility function. The authors there discuss the so-called {\em monotone case} and give conditions for the optimality of {\em one-step-look-ahead rules}. Of course the stopping problems we treat here can be seen as special partially observable risk-sensitive Markov Decision Processes. The general theory for these type of problems has been developed in \cite{br15}. The special case of an exponential utility combined with partial observation has been treated in \cite{cchh05}, \cite{dims99}, \cite{jjr94}, \cite{s04}. The first treatment of the exponential certainty equivalent as an optimization criterion can be traced back to \cite{hm72}. See also \cite{whi87} and \cite{w90} for other risk sensitive criteria like the variance and \cite{br14} for general certainty equivalents with complete observation. Applications to risk management can be found in \cite{dl14}. In \cite{cchh15} the authors discuss the influence of the utility function on the optimal strategy.

Our paper is organized as follows. In the next section we present a short introduction to certainty equivalents and discuss their properties. In particular we show that they can serve as a reasonable criterion for risk-sensitive decision making. In Section \ref{BRsec:stop} we introduce our risk-sensitive stopping problem with partial observation and with finite time horizon. We prove the existence of an optimal stopping time and give a recursive algorithm for the computation of the maximal expected utility. We are also able to show that the more risk averse a decision maker is, the later she will stop. In Section \ref{BRsec:house} we will consider the Bayesian, risk-sensitive house selling problem. Here we generalize the results in \cite{m00} to the Bayesian case. In particular we show the existence of so-called {\em reservation levels} which characterize the optimal stopping strategy. We also discuss the dependence of the reservation levels on the available information, the attitude towards risk of the decision maker and the time horizon. Finally in Section \ref{BRsec:ITH} we consider the general stopping problem with infinite time horizon.

%%%%%%%%%%%%%%%%%%%%%%%%%%%%%%%%%%%%%%%%%%%%%%%%%%%%%%%%%%%%%%%%%%%%%%%%%%%
\section{Certainty Equivalents}
In this section we briefly recall the properties of our objective function and demonstrate that it is very well-suited for risk-sensitive optimization. For a deeper investigation we refer the reader to \cite{m07}. In what follows we work with the space of all bounded real-valued random variables $\mathcal{L}^\infty$. We interpret the outcome of a random variable $X$ as the realization of a risky position. The event $X<0$ represents a loss whereas $X>0$ represents a gain. For a strictly increasing function $U:\RR\to\RR$
%, where $I\subset \RR$ is a suitable interval (which may be unbounded, e.g. $I=(a,\infty), I=\RR$ etc),
we define the {\em certainty equivalent} $\rho_U : \mathcal{L}^\infty\to \RR$ by
$$ \rho_U(X) := U^{-1}\big( \mathbb{E}[U(X)]\big).$$
The quantity is called certainty equivalent since $U(\rho_U(X)) = \mathbb{E}[U(X)]$ which means that the utility of $\rho_U(X)$ is the same as the expected utility of $X$. Obviously $\rho_U$ is {\em law-invariant}, i.e. $\rho_U(X)=\rho_U(Y)$ if $X$ and $Y$ have the same distribution. In what follows we summarize some properties of $\rho_U$. For this purpose we assume that $U$ is {\em strictly increasing and concave}.  Recall that for two real-valued random variable $X$ and $Y$ we say that
\begin{itemize}
  \item[$\bullet$] $X\le_{st} Y$ if and only if $\mathbb{E} f(X) \le \mathbb{E} f(Y)$ for all increasing $f:\RR\to \RR$ for which the expectations exist,
  \item[$\bullet$] $X\le_{cv} Y$ if and only if $\mathbb{E} f(X) \le \mathbb{E} f(Y)$ for all concave $f:\RR\to \RR$ for which the expectations exist.
\end{itemize}
It is now easy to see that $ \rho_U$ is {\em monotone}, i.e.
\begin{itemize}
  \item[$\bullet$] {\bf (monotonicity)} $X\le_{st}Y$ implies that $ \rho_U(X) \le  \rho_U(Y)$.
\end{itemize}
Obviously we also have
\begin{itemize}
  \item[$\bullet$] {\bf (constancy)} $ \rho_U(m)=m$ for any constant $m\in \RR$.
\end{itemize}
 Moreover, we have consistency w.r.t.\ the concave ordering. This property is also called {\em Schur-concavity} (see e.g. \cite{da05}), i.e.
\begin{itemize}
  \item[$\bullet$] {\bf (Schur-concavity)} $X\le_{cv}Y$ implies that $ \rho_U(X) \le  \rho_U(Y)$.
\end{itemize}
%If the consistency is w.r.t. the convex ordering we speak of {\em Schur-convexity}.
Note that Schur-concavity is not the same as {\em concavity of a risk measure $\rho:\mathcal{L}^\infty\to \RR$ } which is defined by $$\rho(\alpha X +(1-\alpha) Y) \ge \alpha \rho(X)+(1-\alpha) \rho(Y),\mbox{ for all } \alpha\in [0,1] \mbox{ and r.v. } X,Y\in \mathcal{L}^\infty.$$ It has been shown in \cite{bm06} that any concave risk measure on a non-atomic probability space is Schur-concave. However, only when $U$ is linear or exponential, the certainty equivalent $\rho_U$ is also concave (see \cite{m07}).

Hence maximizing the certainty equivalent means to prefer risky positions with higher reward and lower variance. In the classical risk neutral setting, the variance of a position does not enter the decision criterion. This may sometimes lead to a very risky behavior (as e.g. demonstrated in \cite{br14} for a casino game). Finally let us consider the following axioms which are often imposed on a risk measure $\rho:\mathcal{L}^\infty\to \RR$:
\begin{itemize}
  \item[$\bullet$] {\bf (translation-invariance)} $\rho(X+m)=\rho(X)+m$ for any constant $m\in\RR$,
  \item[$\bullet$] {\bf (positive homogeneity)} $\rho(\alpha X)=\alpha \rho(X)$ for any constant $\alpha >0$.
\end{itemize}
It has been shown in \cite{m07} that $\rho_U$ is translation-invariant essentially if and only if $U$ is linear or exponential and $\rho_U$ is positive homogeneous if and only if $U$ is the logarithm or a power function (including the linear function).

More general there exists an axiomatic characterization of certainty equivalents which is know as {\em Nagumo-Kolmogorov-de Finetti theorem}. It states that a functional $\rho : \mathcal{L}^\infty \to \RR$ is a certainty equivalent if and only if it is monotone, law invariant, quasi-linear and has the constancy property. For an early proof see \cite{hardy}. For a more recent treatment see \cite{m07}.

Another motivation for certainty equivalents is the fact that they may be written (assuming enough regularity of $U$) as in \eqref{BReq:Urep}.
%\begin{equation}\label{eq:Urep} U^{-1}\big(\mathbb{E}[U(X)]\big)\approx \mathbb{E} X - \frac12 l_U(\mathbb{E} X) Var[X]\end{equation}
%where \begin{equation}\label{eq:AP} l_U(x) = -\frac{U''(x)}{U'(x)}\end{equation} is the {\em Arrow-Pratt} function of absolute risk aversion.
Thus the certainty equivalent is approximately a weighted criterion of expectation and variance.

An important special case is obtained when $$U(x)=\frac1\gamma e^{\gamma x},\, x\in\RR, \mbox{ with } \gamma <0$$
is an exponential function. In this case $$\rho_U(X) = \frac1\gamma \ln \mathbb{E} e^{\gamma X}$$ is the {\em entropic risk measure}. In this case if $X\sim \mathcal{N}(\mu,\sigma^2)$ then we obtain by direct calculation that $\rho_U(X) = \mu +\frac12 \gamma \sigma^2.$ The entropic risk measure also has a dual representation given by
$$ \rho_U(X) = \inf_{\mathbb{Q}\ll \PP} \Big\{\mathbb{E}_\mathbb{Q}[X] -\frac1\gamma \mathbb{E}_\mathbb{Q}\Big[\ln \frac{d\mathbb{Q}}{d\PP}\Big]\Big\},$$
where the infimum is attained at $$\mathbb{Q}^*(dz) = \frac{e^{\gamma z} \PP(dz)}{\int e^{\gamma y} \PP(dy)}.$$
Thus we can interpret the optimization problem in this case as a game against nature where the second player (nature) chooses the probability measure.

%%%%%%%%%%%%%%%%%%%%%%%%%%%%%%%%%%%%%%%%%%%%%%%%%%%%%%%%%%%%%%%%%%%%%%%%%%%%%%%%
\section{Risk-sensitive Stopping Problems}\label{BRsec:stop}
%%%%%%%%%%%%%%%%%%%%%%%
We suppose that a {\em partially observable Markov Process} is given which we introduce as follows: We denote this process by $(X_n,Y_n)_{n\in\NN_0}$ and assume a Borel state space $E_X\times E_Y$. The $x$-component will be the {\em observable} part, the $y$-component {\em cannot be observed} by the controller. There is a stochastic transition kernel $Q$ from $E_X\times E_Y$ to $E_X\times E_Y$ which determines the distribution of the new state pair given the current state. So $Q(B|x,y)$ is the probability that the next state pair is in $B\in\mathcal{B}(E_X\times E_Y)$, given the current state is $(x,y)$. In what follows we assume that the transition kernel $Q$ has a density $q$ with respect to some $\sigma$-finite measures $\lambda$ and $\nu$, i.e.
$$ Q(B | x,y) =\int_{B} q(x',y'|x,y) \lambda(dx') \nu(dy'),\quad B\in\mathcal{B}(E_X\times E_Y).$$
For convenience we introduce the marginal transition kernel density by
$$ q^X(x'|x,y) := \int_{E_Y} q(x',y'|x,y) \nu(dy').$$
We assume that the initial distribution $Q_0$ of $Y_0$ is known. For a fixed (observable) initial state $x\in E_X$, the initial
distribution $Q_0$ together with the transition probability $Q$ define by the Theorem of Ionescu Tulcea a probability measure $\PP_{x}$ on
$(E_X\times E_Y)^{N+1} $ endowed with the product $\sigma$-algebra. More precisely $\PP_{xy}$ is the probability measure given $X_0=x$ and $Y_0=y$ and we define $\PP_x(\cdot) := \int P_{xy}(\cdot) Q_0(dy)$.

Next we have a measurable one-stage reward function $c: E_X\to \RR$ which depends only on the observable part of the process. Since $c<0$ is possible this could also be a cost. This reward (cost) is obtained as long as the process is not stopped. A reward $g:E_X\to\RR$ is obtained when the process is stopped. Stopping-times are taken w.r.t. the observable filtration $(\mathcal{F}_n)$ where $\mathcal{F}_n = \sigma(X_0,\ldots ,X_n)$. We first consider  problems with a finite time horizon $N$ and maximize the certainty equivalent of the stopping reward. Since $U^{-1}$ is increasing we can skip it from the optimization problem. Thus, we define
\begin{eqnarray}\label{BReq:S}
% \nonumber to remove numbering (before each equation)
    J_N(x) &:=& \sup_{0\le \tau \le N}  \mathbb{E}_{x}\Big[U\Big(\sum_{k=0}^{\tau-1} c(X_k)+g(X_{\tau}) \Big)   \Big].
\end{eqnarray}
In order to have a well-defined problem we assume that for all $x\in E_X$
\begin{eqnarray}\label{BReq:int}
% \nonumber to remove numbering (before each equation)
     \sup_{0\le \tau \le N}  \sup_{y\in E_Y} \mathbb{E}_{xy}\Big[\sum_{k=0}^{\tau-1} c^+(X_k)+g^+(X_{\tau}) \Big]<\infty.
\end{eqnarray}
Note that since $U$ is concave it can be bounded by a linear function and thus \eqref{BReq:int} implies that $J_N(x)<\infty$ for all $x\in E_X$. Unfortunately \eqref{BReq:S} cannot be solved with classical MDP techniques directly. Thus, we introduce the following auxiliary problems. For a probability measure $\mu\in \PP(E_Y)$ and a constant $s\in\RR$ define for $n=1,\ldots,N$:
\begin{eqnarray}
% \nonumber to remove numbering (before each equation)
\label{BReq:V-def}  V_{n}(x,\mu,s)  &:=& \sup_{0\le \tau \le n}  \int_{E_Y} \mathbb{E}_{xy} \left[ U\Big(\sum_{k=0}^{\tau-1} c(X_k)+s+g(X_{\tau})\Big)\right]\mu(dy)
\end{eqnarray}
In view of \eqref{BReq:int} $V_n(x,\mu,s) < \infty$ for all $(x,\mu,s)$.
Obviously we have by this embedding technique that $J_N(x) = V_N(x,Q_0,0)$. We claim now that \eqref{BReq:V-def} can be solved with MDP techniques. In order to do so, we first have to cope with the fact that $(Y_n)$ is not observable. Hence we define the operator $\Phi:E_X\times E_X\times \PP(E_Y) \to \PP(E_Y)$ by
\begin{equation*}
 \Phi(x,x',\mu)(B) :=   \frac{\int_{B} q^X(x'|x,y) \mu(dy)}{\int_{E_Y} q^X(x'|x,y) \mu(dy) },\; B\in\mathcal{B}(E_Y).
\end{equation*}
Note that $\Phi$ is exactly the usual updating (Bayesian) operator which appears in classical POMDP (see e.g. \cite{br11}, Section 5.2). It updates the conditional probability of the unobservable state. In what follows denote by $(\mu_n)$ the sequence of probability measures on $E_Y$ generated by $\Phi$ with $\mu_0 := Q_0$. I.e. for $n\in\NN$, a history  $h_n := (x_0,\ldots ,x_n)\in E_X^{n+1}$ and $B\in \mathcal{B}(E_Y)$ we  define
\begin{eqnarray}
% \nonumber to remove numbering (before each equation)
\nonumber  \mu_0(B|h_0) &:=& Q_0(B), \\
\label{BReq:murec1}  \mu_{n+1}(B|h_{n},x')&:=& \Phi\big(x_n,x',\mu_n(\cdot|h_n)\big)(B)
\end{eqnarray}
Then it is well-known (see e.g. \cite{br11}, Theorem 5.2.1) that
$$ \mu_n(B|X_0,\ldots,X_n) = \PP_x\big(Y_n\in B|X_0,\ldots,X_n\big),\quad B\in \mathcal{B}(E_Y).$$

We consider now a stopping problem with a process on the state space ${E}:=E_X\times \PP(E_Y)\times \RR$. The running reward is zero and the stopping reward  is $ g(x,\mu,s):= U\big(g(x)+s\big)$.  The transition law of the process is given by $\tilde{Q}(\cdot|x,\mu,s)$ which is for $(x,\mu,s) \in E$ and a measurable function $v:E\to \RR$ defined by
\begin{eqnarray*}
  && \int_{E_X}\int_{ \PP(E_Y)}\int_{\RR} v(x',\mu',s') \tilde{Q}(d(x',\mu',s') | x,\mu,s) \\
  && \hspace*{2cm}=\int_{E_X} v\Big(x',\Phi(x,x',\mu),s+c(x)\Big) Q^X(dx'|x,\mu)
\end{eqnarray*}
where \begin{equation}\label{BReq:QX} Q^X(B|x,\mu) := \int_B\int_{E_Y} q^X(x'|x,y) \mu(dy) \lambda(dx').\end{equation}
Stopping times are now considered w.r.t. the filtration $(\mathcal{G}_n)$ which is  defined by\linebreak  $\mathcal{G}_n = \sigma(X_0,\mu_0,S_0,\ldots ,X_n,\mu_n,S_n)$ with $S_n := \sum_{k=0}^{n-1} c(X_k)$. However note that obviously by construction of the sequences $(\mu_n)$ and $(S_n)$ we obtain  $\mathcal{F}_n =\mathcal{G}_n$.

\begin{theorem}\label{BRtheo:Bellman1}
It holds:
\begin{itemize}
  \item[a)] For $n=1,\ldots ,N$ and $(x,\mu,s) \in E_X\times\PP(E_Y)\times \RR$, the functions $V_n$ are given by
  \begin{eqnarray}
  % \nonumber to remove numbering (before each equation)
\nonumber   && V_0(x,\mu,s) = U\big(g(x)+s\big) \\
  \nonumber  && V_n(x,\mu,s) =\\
\label{BReq:vn}    && \max\Big\{ U\big(g(x)+s\big), \int_{E_X} V_{n-1}\Big(x',\Phi(x,x',\mu),s+c(x)\Big) Q^X(dx'|x,\mu)\Big\}
  \end{eqnarray}
The value function of \eqref{BReq:S} is then given by $J_N(x)=V_N(x,Q_0,0)$.
  \item[b)] For every $n=1,\ldots ,N$ and $(x,\mu,s) \in E_X\times\PP(E_Y)\times \RR$, let $f_n^*(x,\mu,s)=1$ if the maximum in the recursion \eqref{BReq:vn} is taken at $U\big(g(x)+s\big)$ and define $(g_0^*,\ldots,g_{N-1}^*)$ by
 $$ g_n^*(h_n) := f_{N-n}^*\big( x_n, \mu_n(\cdot|h_n),\sum_{k=0}^{n-1} c(x_k)\big),\quad n=0,\ldots,N-1.$$
 Then the optimal stopping time is given by
 $$ \tau^* := \inf\{ n\in\NN_0 : g_n^*(h_n) = 1\}\wedge N.$$
 Note that $\tau^*$ depends on the history $h_n$ of the process.
  \end{itemize}
\end{theorem}

\begin{proof}
The proof follows from Theorem 3.3 in \cite{br15}.%\hfill\blacksquare
\end{proof}
%\begin{remark}
%Of course the problem can also be solved by looking at it as a standard partially observable non-stationary stopping problem of the form \linebreak $\sup_{0\le \tau \le N}  \int_{E_Y} \mathbb{E}_{xy}[Z_\tau] Q_0(dy)$ with $Z_\tau = U\Big(\sum_{k=0}^{\tau-1} c(X_k)+g(X_{\tau}) \Big)$. For details see \cite{crs71}.
%\end{remark}

\begin{remark}
By definition we have $U\big(g(x)+s\big)\le V_n(x,\mu,s) \le V_{n+1}(x,\mu,s)$ for all $n\in\NN$ and  $(x,\mu,s) \in E_X\times\PP(E_Y)\times \RR$ i.e. the value of the stopping problem increases with the time horizon. This observation is of course obvious, since all stopping times which are feasible for a smaller time horizon are also feasible for a larger time horizon. Note that we use increasing and decreasing in a non-strict sense throughout.
\end{remark}

\begin{example}
Let $U(x) = \frac1\gamma e^{\gamma x}$ with $\gamma <0$. In this case Theorem \ref{BRtheo:Bellman1} simplifies due to the positive homogeneity of $\rho_U$ and it is easy to see that $$V_n(x,\mu,s)=e^{\gamma s} h_n(x,\mu)$$
and the $h_n$ satisfy the recursion
\begin{eqnarray*}
% \nonumber to remove numbering (before each equation)
  h_0(x,\mu) &=& \frac1\gamma e^{\gamma g(x)} \\
  h_n(x,\mu) &=& \max\Big\{ \frac1\gamma e^{\gamma g(x)}, e^{\gamma c(x)} \int_{E_X} h_{n-1}\Big(x',\Phi(x,x',\mu)\Big) Q^X(dx'|x,\mu)\Big\}.
\end{eqnarray*}
As a result, in the exponential case the state space is reduced by one variable and the state of the auxiliary problem consists only of $x$ and the conditional distribution of $y$ given the history $h_n$. In particular the optimal stopping rule depends on the history only through the conditional probability of the unobservable random variable. The same situation arises in the classical risk-neutral stopping problem.
\end{example}

Finally we discuss the influence of the risk attitude of the decision maker on the optimal stopping time. We use the Arrow-Pratt function of absolute risk aversion \eqref{BReq:AP} to measure the risk sensitivity. A utility function $U$ is said to be {\em more risk averse} than a utility function $W$ if $l_U(x) \ge l_W(x)$ for all $x\in\RR$. For our purpose it is crucial to note that a utility function $U$ is more risk averse than a utility function  $W$ if and only if, there exits an increasing concave function $r:\RR\to\RR$ such that $U=r\circ W$. In what follows we denote all quantities which refer to utility function $U$ by $g_n^*(h_n,U), V_n(x,\mu,s,U)$ etc. and similar for $W$.

\begin{theorem}\label{BRtheo:compstoppingrule}
Suppose that the utility function $U$ is more risk averse than the utility function  $W$. For all $n=0,1,\ldots,N-1$ and histories $h_n$ we obtain that $g_n^*(h_n,W) =1 $ implies $g_n^*(h_n,U)=1$, i.e. a more risk-averse decision maker will not stop later.
\end{theorem}

\begin{proof}
Let $r$ be such that $U=r\circ W$. We first prove by induction on $n$ that $V_n(x,\mu,s,U)\le r\circ V_n(x,\mu,s,W)$ for all $(x,\mu,s)$ and $n$.
First for $n=0$ we have
\begin{equation*}
    V_0(x,\mu,s,U) = U\big( g(x)+s\big) = r\circ W\big( g(x)+s\big) = r\circ V_0(x,\mu,s,W).
\end{equation*}
Using the Jensen inequality, the induction hypothesis and the fact that $r$ is increasing and concave we obtain
\begin{eqnarray*}
% \nonumber to remove numbering (before each equation)
  &&V_n(x,\mu,s,U) =\\
  &=& \max\Big\{ U\big(g(x)+s\big), \int_{E_X} V_{n-1}\Big(x',\Phi(x,x',\mu),s+c(x),U\Big) Q^X(dx'|x,\mu)\Big\}\\
   &\le& \max\Big\{ r\circ W\big(g(x)+s\big), \int_{E_X} r\circ V_{n-1}\Big(x',\Phi(x,x',\mu),s+c(x),W\Big) Q^X(dx'|x,\mu)\Big\} \\
   &\le& \max\Big\{ r\circ W\big(g(x)+s\big), r\circ \int_{E_X}  V_{n-1}\Big(x',\Phi(x,x',\mu),s+c(x),W\Big) Q^X(dx'|x,\mu)\Big\} \\
   &=& r\circ \max\Big\{ W\big(g(x)+s\big), \int_{E_X}  V_{n-1}\Big(x',\Phi(x,x',\mu),s+c(x),W\Big) Q^X(dx'|x,\mu)\Big\} \\
   &=& r\circ  V_n(x,\mu,s,W).
\end{eqnarray*}
This implies in particular that $$W\big(g(x)+s\big) \ge \int_{E_X}  V_{n-1}\Big(x',\Phi(x,x',\mu),s+c(x),W\Big) Q^X(dx'|x,\mu)$$
leads to
\begin{eqnarray*}
% \nonumber to remove numbering (before each equation)
&& U\big(g(x)+s\big)=\\
&=&  r\circ W\big(g(x)+s\big) \ge r\circ \int_{E_X}  V_{n-1}\Big(x',\Phi(x,x',\mu),s+c(x),W\Big) Q^X(dx'|x,\mu) \\
   &\ge&  \int_{E_X} r\circ   V_{n-1}\Big(x',\Phi(x,x',\mu),s+c(x),W\Big) Q^X(dx'|x,\mu)  \\
   &\ge&  \int_{E_X}   V_{n-1}\Big(x',\Phi(x,x',\mu),s+c(x),U\Big) Q^X(dx'|x,\mu).
\end{eqnarray*}
By definition this means that $f_n^*(x,\mu,s,W)=1$ implies $f_n^*(x,\mu,s,U)=1$. Transferring this observation to $g_n^*$ implies the statement.
\end{proof}

\section{Risk-Sensitive Bayesian House Selling Problem}\label{BRsec:house}
Suppose someone wants to sell her house. We assume that offers $X_0,\ldots,X_N$ arrive independently and are identically distributed with distribution $Q_\theta$.  Here $\theta\in \Theta$ is an unknown parameter. We assume that $Q_\theta$ has a density $q(x|\theta)$ and some prior distribution $Q_0$ for $\theta$ is given. As long as offers are rejected a maintenance cost of $c>0$ has to be paid. When an offer is accepted, the price is obtained and the game ends. The aim is to find the maximal risk-sensitive stopping reward
\begin{eqnarray}\label{BReq:H}
% \nonumber to remove numbering (before each equation)
    J_N(x) &:=& \sup_{0\le \tau \le N}  \int_{\Theta} \mathbb{E}_{x\theta}\Big[U\Big(X_{\tau}-c \tau \Big)   \Big] Q_0(d\theta).
\end{eqnarray}
This is a special case of our general model with $Y_n\equiv \theta$, $c(x)=-c$ and $g(x)=x$. The integrability condition \eqref{BReq:int} reduces to
\begin{equation*}
    \sup_\theta \mathbb{E}_{\theta}[X_1^+] <\infty.
\end{equation*}
Note that this problem without partial observation has been investigated in \cite{m00}. According to Theorem \ref{BRtheo:Bellman1} we obtain $J_N$ by computing the functions $V_n$. These are given by
 \begin{eqnarray*}
  % \nonumber to remove numbering (before each equation)
    V_0(x,\mu,s) &=& U\big(x+s\big) \\
    V_n(x,\mu,s) &=& \max\Big\{ U\big(x+s\big), c_n(\mu,s) \Big\}
  \end{eqnarray*}
with $ c_n(\mu,s):=\int_{\RR} V_{n-1}\Big(x',\Phi(x',\mu),s-c\Big) Q(dx'|\mu)$ and we have that $J_N(x)=V_N(x,Q_0,0)$. Note that $Q(dx'|\mu)$ corresponds to $Q^X(dx'|x,\mu)$ in \eqref{BReq:QX} and is thus given by
$$ Q(\cdot|\mu) = \int_\Theta Q_\theta(\cdot) \mu(d\theta)$$ and the updating operator simplifies to
$$  \Phi(x,\mu)(B) :=   \frac{\int_{B} q(x|\theta) \mu(d\theta)}{\int_{\Theta} q(x|\theta) \mu(d\theta) },\; B\in\mathcal{B}(\Theta).
$$ Moreover, when we define $f_n^*(x,\mu,s)=1$ if $U\big(x+s\big)\ge c_n(\mu,s)$ and $(g_0^*,\ldots,g_{N-1}^*)$ by
 $$ g_n^*(h_n) := f_{N-n}^*\big( x_n, \mu_n(\cdot|h_n),-nc\big),\quad n=0,\ldots,N-1,$$
then the optimal stopping time for problem \eqref{BReq:H} is given by
 $$ \tau^* := \inf\{ n\in\NN_0 : g_n^*(h_n) = 1\}\wedge N.$$
 Let us now further investigate the optimal stopping time. We have
 $$ g_n^*(h_n)=1 \quad \Leftrightarrow \quad U(x_n-nc) \ge c_{N-n}(\mu_n(\cdot|h_n),-nc).$$
When we define
 \begin{eqnarray*}
 % \nonumber to remove numbering (before each equation)
   U_n(x) &:=& U(x-nc) \\
   d_k(\mu,U_n) &:=& c_k(\mu,-nc) = \int_{\RR} V_{k-1}\Big(x',\Phi(x',\mu),-(n+1)c\Big) Q(dx'|\mu),
 \end{eqnarray*}
 then we obtain
 $$ g_n^*(h_n)=1 \quad \Leftrightarrow \quad x_n \ge U_n^{-1}\big(d_{N-n}(\mu_n(\cdot|h_n),U_n)\big) =: x_{n,N}^*(\mu_n(\cdot|h_n)).$$
  Note that $U_n^{-1}(x) = nc +U^{-1}(x)$. We call $x_{n,N}^*(\cdot)$ {\em reservation level}. The reservation levels depend on $\mu_n$ and $U$. The optimal stopping time is hence the first time, the offer exceeds the corresponding, history dependend reservation level.

\begin{theorem}\label{BRtheo:xrec}
\begin{itemize}
  \item[a)] The optimal stopping time for the Bayesian house selling problem is given by
$$\tau^* = \inf\big\{ n\in\NN_0 : X_n \ge x_{n,N}^*(\mu_n(\cdot|h_n))\big\}\wedge N.$$
  \item[b)] The reservation levels can recursively be computed by
  \begin{eqnarray*}
  % \nonumber to remove numbering (before each equation)
    x_{N-1,N}^*(\mu) &=& U_{N-1}^{-1} \circ \int_\RR  U_N(x) Q(dx|\mu)   \\
    x_{n,N}^*(\mu) &=&  U_{n}^{-1} \circ \int_\RR U_{n+1}\Big( \max\big\{ x, x_{n+1,N}^*(\Phi(x,\mu))\big\}\Big)  Q(dx|\mu).\qed
  \end{eqnarray*}
\end{itemize}
\end{theorem}

\begin{proof}
Part a) is clear from the definition and the previous results. Part b) can be shown by inserting the correct definitions. For $n=N-1$ we obtain from the definition of $x_{N-1,N}^*(\mu) $ that $$x_{N-1,N}^*(\mu) = U_{N-1}^{-1} \big( d_1(\mu,U_{N-1})\big)$$ with $$ d_1(\mu, U_{N-1})= \int V_0\big(x,\Phi(x,\mu),-Nc\big) Q(dx|\mu) = \int U_N(x) Q(dx|\mu).$$
For $ x_{n,N}^*$ we obtain by definition:
$$ x_{n,N}^*(\mu)= U_n^{-1}\big(d_{N-n}(\mu,U_n)\big).$$
Further $d_{N-n}(\mu,U_n)$ can be written as
\begin{eqnarray*}
% \nonumber to remove numbering (before each equation)
&&  d_{N-n}(\mu,U_n) =  \int_{\RR} V_{N-n-1}\Big(x,\Phi(x,\mu),-(n+1)c\Big) Q(dx|\mu)\\
   &=& \int_{\RR} \max\Big\{ U(x-(n+1)c), \\
   &&\int V_{N-n-2}\Big(x',\Phi(x',\Phi(x,\mu)),-(n+2)c\Big) Q(dx'|\Phi(x,\mu))\Big\} Q(dx|\mu) \\
   &=& \int_{\RR} \max\Big\{ U_{n+1}(x), d_{N-n-1}(\Phi(x,\mu),U_{n+1})\Big\}Q(dx|\mu)\\
 &=& \int_{\RR} U_{n+1}\Big(\max\big\{ x, U_{n+1}^{-1}\circ d_{N-n-1}(\Phi(x,\mu),U_{n+1})\big\}\Big)Q(dx|\mu)
\end{eqnarray*}
and the statement follows from the definition of $x_{n+1,N}^*$.
\end{proof}

\begin{remark}
It obviously holds that
\begin{eqnarray*}
% \nonumber to remove numbering (before each equation)
  d_N(\mu,U) &=&  \int_{\RR} V_{N-1}\Big(x,\Phi(x,\mu),-c\Big) Q(dx|\mu) \\
   &=& \sup_{1\le \tau \le N} \int_\Theta \int_{\RR} \mathbb{E}_{x\theta} \big[ U(X_\tau-c\tau)\big] Q_\theta(dx)\mu(d\theta),
\end{eqnarray*}
i.e. $d_N(Q_0,U)$ is the value of the stopping problem when we start without known initial offer.
\end{remark}

\begin{example}\label{BRex:num}
In case $U(x) = \frac1\gamma e^{\gamma x}$ with $\gamma <0$ the recursion for the reservation levels simplifies. In order to see this, note that $U_n(x) = U(x-nc)=\frac1\gamma e^{\gamma x} \cdot e^{-\gamma nc}$ and $U_n^{-1}(x) = U^{-1}(x)+nc = \frac1\gamma \ln(\gamma x) +nc$. With these observation we obtain
\begin{eqnarray*}
  % \nonumber to remove numbering (before each equation)
    x_{N-1,N}^*(\mu) &=& -c+\frac1\gamma \ln \Big( \int_\RR e^{\gamma x} Q(dx|\mu)\Big)   \\
    x_{n,N}^*(\mu) &=&  -c+ \frac1\gamma \ln\Big( \int_\RR e^{\gamma \max\big\{ x, x_{n+1,N}^*(\Phi(\mu,x))\big\}}  Q(dx|\mu)\Big).
  \end{eqnarray*}
Note that in contrast to the $  x_{n,N}^*$ with general utility function, the reservation levels in the exponential utility case depend only on the time difference to the planning horizon. More precisely we could also define $x_n^*(\mu) := x_{N-n,N}^*(\mu)$ and obtain
 \begin{eqnarray*}
  % \nonumber to remove numbering (before each equation)
    x_{1}^*(\mu) &=& -c+\frac1\gamma \ln \Big( \int_\RR e^{\gamma x} Q(dx|\mu)\Big)   \\
    x_{n}^*(\mu) &=&  -c+ \frac1\gamma \ln\Big( \int_\RR e^{\gamma \max\big\{ x, x_{n-1}^*(\Phi(\mu,x))\big\}}  Q(dx|\mu)\Big).
  \end{eqnarray*}
  In particular we can start computing the reservation levels without fixing a planning horizon in advance.

  We did this in the following {\em numerical example}, where $Q_\theta = B(1,\theta)$ is a Bernoulli distribution with unknown 'success' probability $\theta$, the prior distribution of $\theta$ is uniform on $[0,1]$ and $c= 0.1$. For a planning horizon of $N=10$ we computed the optimal stopping rule. Of course an offer of $1$ will always be accepted. An offer of $0$ may be accepted when we fear the accumulation of cost. It turns out that this decision heavily depends on the risk aversion parameter $\gamma<0$. The smaller $\gamma$, the more risk averse the decision maker is.
  \begin{figure}
    \centering
   \begin{tabular}{|c|c|c|c|c|c|c|c|}
    \hline
    % after \\: \hline or \cline{col1-col2} \cline{col3-col4} ...

    \rule[-4mm]{0mm}{10mm}$-\infty< \gamma < -2.2$ & $-2.2<\gamma < -1.51$ & $-1.51<\gamma < -1.1$ & $-1.1<\gamma <  -0.8$   \\

    \hline
    \rule[-2mm]{0mm}{6mm} 0 & 1 & 2 & 3   \\ \hline
 %    \end{tabular}\\[0.4cm]
 %   \begin{tabular}{|c|c|c|c|c|c|c|c|}
    \hline
     \rule[-4mm]{0mm}{10mm}   $-0.8<\gamma< -0.56$   &  $-0.56<\gamma <-0.34$ & $-0.34<\gamma < -0.18$ & $-0.18<\gamma < -0.03$ \\ \hline
     \rule[-2mm]{0mm}{6mm} 4 & 5 & 6 & 7 \\\hline
     \end{tabular}
   \caption{Optimal number of zeros which are rejected.}
    \label{BRfig:w}
\end{figure}
 {\bf Figure \ref{BRfig:w}} has to be read as follows: When $\gamma < -2.2$, the decision maker will stop immediately. For $-2.2<\gamma < -1.51$ she will reject at least the first offer when it is zero. For  $-1.51<\gamma < -1.1$ she will reject the first two offers when they are zero. And so on. The switch from rejecting 8 to 9 zeros is for $\gamma$ smaller than $-10^{-8}$.

\end{example}

\subsection{Influence of the Filter on the Reservation Levels}
In order to discuss the influence of the filter on the reservation levels we make some further simplifying assumptions. Indeed it is often the case that the filter $\mu_n(\cdot |h_n)$ does only depend on a part of the history $h_n$ or on a certain function of it. In what follows we assume that there is a (Borel) information set $I$ endowed with a $\sigma$-algebra such that there exist a measurable function $t_n : H_n \to I$ and a transition kernel $\hat{\mu}$ from $I$ to $\Theta$ such that
$$ \mu_n(B|h_n) = \hat{\mu}(B|t_n(h_n)),\quad \mbox{for } B\in \mathcal{B}(E_Y).$$
The function $t_n$ is sometimes called {\em sufficient statistics}. Further we assume that there is a measurable mapping $\hat{\Phi}: I \times E_X$ such that
$$ t_{n+1}(h_{n+1}) = \hat{\Phi}\big( x_{n+1},t_n(h_n)\big)$$
and that $\hat{\mu}$ has a density $\hat{p}(\cdot|i)$.
Thus, we have to replace $\mu_n$ by the current information state and obtain in particular for the reservation levels:
  \begin{eqnarray*}
  % \nonumber to remove numbering (before each equation)
    x_{N-1,N}^*(i) &:=& U_{N-1}^{-1} \circ \int_\RR  U_N(x) Q(dx|i)   \\
    x_{n,N}^*(i) &:=&  U_{n}^{-1} \circ \int_\RR U_{n+1}\Big( \max\big\{ x, x_{n+1,N}^*(\hat{\Phi}(x,i))\big\}\Big)  Q(dx|i)
  \end{eqnarray*}
where
$$Q(B|i) := \int_B\int_{\Theta} q(x|\theta) \hat{\mu}(d\theta|i) \lambda(dx).$$
The next step is to introduce an order relation on the set $I$ where we assume now that $\Theta\subset \RR$. We define here for $i,i'\in I$
\begin{equation*}
    i \le i' \quad :\Leftrightarrow \quad \hat{\mu}(\cdot |i) \le_{lr} \hat{\mu}(\cdot |i')
\end{equation*}
where $\le_{lr} $ is the likelihood ratio ordering which is defined by
\begin{equation*}
   \hat{\mu}(\cdot |i) \le_{lr} \hat{\mu}(\cdot |i') \quad \Leftrightarrow\quad  \frac{\hat{p}(\theta|i')}{\hat{p}(\theta|i)}\quad\mbox{is increasing in } \theta.
\end{equation*}
Note that $\hat{p}$ is the density of $\hat{\mu}$. The likelihood ratio ordering implies the stochastic ordering.  Now we are able to formulate the main result of this subsection

\begin{theorem}\label{BRlem:info}
Suppose that $q(x|\theta)$ is $MTP_2$, i.e. $q(\cdot|\theta) \le_{lr} q(\cdot |\theta')$ for all $\theta\le \theta'$, then the reservation levels $x^*_{n,N}(i)$ are increasing in $i$.
\end{theorem}

\begin{proof}
We prove the statement by induction on $n$. First consider
$$  x_{N-1,N}^*(i) = U_{N-1}^{-1} \circ \int_\RR  U_N(x)\int_\Theta q(x|\theta)\hat{\mu}(d\theta|i)\lambda(dx). $$
Since $x\mapsto U_N(x)$ is increasing and the $\le_{lr}$ implies the $\le_{st}$ order we obtain that
$$ \theta \mapsto  \int_\RR  U_N(x) q(x|\theta) \lambda(dx) =: f(\theta)$$
is increasing in $\theta$. Thus by the definition of $i\le i'$ we obtain that $\int_\Theta f(\theta)\hat{\mu}(d\theta|i)$ is increasing in $i$. Now suppose the statement is true for $x_{n+1,N}^*(i)$. Next note that by our assumption on $q(x|\theta)$ we have that $$ (i,x) \mapsto \hat{\Phi}(x,i)$$ is increasing. This follows from Lemma 5.4.9 in \cite{br11}. Hence
$$ (i,x)\mapsto U_{n+1}\Big( \max\big\{ x, x_{n+1,N}^*(\hat{\Phi}(x,i))\big\}\Big)=: f(i,x)$$
is increasing. Next by our assumption on $q(x|\theta)$ we have that
$$ (i,\theta) \mapsto \int_\RR f(i,x) q(x|\theta) \lambda(dx)=: \hat{f}(i,\theta)$$
is increasing. And finally we obtain that
$$ x^*_{n,N}(i) = U_n^{-1}\circ \int_\Theta \hat{f}(i,\theta) \hat{\mu}(d\theta|i)$$ is increasing which completes the induction.
\end{proof}

%In Section \ref{subsec:ex} we will consider a special example where all these assumptions are satisfied and the lemma can be applied.

For further details and examples we refer the reader to Section 5.4 in \cite{br11}.

\begin{example}\label{BRsubsec:ex}
In this example we consider the special case of exponentially distributed random variables (offers)
$$ q(x|\theta) = \frac{1}{\theta} e^{-\frac{1}{\theta}x} , \quad x \ge 0,\; \theta\in\Theta:=(0,\infty).$$
According to our definition of $\mu_n$ we get by recursion for $h_n=(x_1,\ldots,x_n)$ and $B\in \mathcal{B}(\Theta)$
$$ \mu_n(B|h_n) = \frac{\int_B \frac{1}{\theta^n} \exp( -\frac1\theta \sum_{k=1}^n x_k) Q_0(d\theta)}{\int_\Theta \frac{1}{\theta^n} \exp( -\frac1\theta \sum_{k=1}^n x_k) Q_0(d\theta)}.$$
Thus $t_n(x_1,\ldots ,x_n) = \Big(\sum_{k=1}^{n} x_k, n\Big)$ is a sufficient statistic. Thus, we have $I:=\RR_+\times \NN_0$ and
denote $i=(s,n)\in I$. Moreover, $\hat{\Phi}(x,(s,n))=(s+x,n+1)$ and the conditional distribution of the unknown parameter has
the form
$$ \hat{\mu}\big(d\theta |s,n\big) \propto \Big( \frac1\theta\Big)^n
e^{-\frac{s}{\theta}} Q_0(d\theta)$$ if the information $(s,n)$ is
given. With this representation it is not difficult to verify that
$$ i=(s,n) \le i'=(s',n') \quad \Leftrightarrow\quad s\le s'
\;\mbox{and}\; n\ge n'.$$ Further, the family of densities $
q(x|\theta)$ is $MTP_2$ in $x\ge 0$ and $\theta\ge 0$, thus Theorem \ref{BRlem:info} applies.

If we assume now a special prior distribution of the unknown parameter then we
obtain an explicit distribution for $Q(\cdot|i)$. We assume that the prior
distribution $Q_0$ is a so-called {\em Inverse Gamma distribution},
i.e.\ the density is given by
$$  Q_0(d\theta) = \frac{b^a}{\Gamma(a)} \Big( \frac1\theta\Big)^{a+1}
e^{-\frac{b}{\theta}} d\theta,\quad \theta>0 $$ where $a>1$ and
$b>0$ are fixed.
%The name relates to the fact that $\frac1\vartheta$
%has a Gamma distribution with parameters $a$ and $b$.
Then the distribution ${Q}$ is given by
$${Q}\big(dx|s,n\big) = \int q(dx|\theta) \hat{\mu}\big(d\theta
|s,n\big) = (n+a) \frac{(s+b)^{n+a}}{(x+s+b)^{n+a+1}}  dx.$$ Hence
${Q}$ is a special {\em Second Order Beta distribution}. In a risk-neutral situation a similar setting has been considered in \cite{t84}.
\end{example}

\subsection{Influence of the Utility Function on the Reservation Levels}
Here we proceed as in \cite{m00} (Theorem 3.3) in order to study the impact of the risk aversion on the reservation levels. As in section \ref{BRsec:stop} we use the Arrow-Pratt function of absolute risk aversion \eqref{BReq:AP} to measure the risk attitude in the sense that a utility function $U$ is more risk averse than a utility function $W$ if $l_U(x) \ge l_W(x)$ for all $x\in \RR$. In what follows we denote by $x_{n,N}^*(\mu,U)$ the reservation levels which belong to the utility function $U$. Then we obtain

\begin{theorem}\label{BRtheo:risk}
If $U$ is a more risk averse utility function than $W$, then the reservation levels satisfy $x_{n,N}^*(\mu,U)\le x_{n,N}^*(\mu,W)$ for all $n=0,1\ldots, N-1$ and all $\mu\in\PP(E_Y)$.
\end{theorem}

\begin{proof}
The proof follows from Theorem \ref{BRtheo:compstoppingrule}.
\end{proof}

\begin{remark}
Theorem \ref{BRtheo:risk} includes as a special case the comparison to the risk neutral stopping problem: Suppose $U$ is an arbitrary increasing concave utility function, then obviously $U=U\circ \rm {id}$. Thus, we choose $r=U$ in this context and see that the reservation levels of a risk neutral decision maker will always be above the reservation levels of a risk averse decision maker.
\end{remark}

\begin{example}\label{BRex:incinga}
In case $U(x) = \frac1\gamma e^{\gamma x}$ with $\gamma <0$ we obtain from Theorem \ref{BRtheo:risk} that the reservation levels $x_n^*(\mu)$ are increasing in $\gamma$.
\end{example}

\subsection{Influence of the Time Horizon on the Reservation Levels}
In a risk neutral setting it is often the case that the reservation levels are decreasing as time goes by, i.e. the decision maker becomes less selective when she approaches the time horizon. However in  \cite{m00} it has been shown that this is no longer true in a risk averse setting. Indeed he gave some examples where the reservation levels are increasing.

Without additional assumptions it is difficult to determine how the reservation levels behave in time.
However in general the reservation levels satisfy the following relation.

\begin{theorem}
For all $n=0,1,\ldots,N-2$ and $\mu\in\PP(E_Y)$ it holds that
$$x_{n,N}^*(\mu)  \ge nc + \rho_U\Big( x_{n+1,N}^*(\Phi(X_1,\mu))-(n+1)c\Big).$$
In case $U(x) = \frac1\gamma e^{\gamma x}$ with $\gamma <0$ we obtain
$$x_{n,N}^*(\mu) +c \ge \rho_U\Big( x_{n+1,N}^*(\Phi(X_1,\mu))\Big).$$
\end{theorem}

\begin{proof}
From Theorem \ref{BRtheo:xrec} b) we obtain
\begin{eqnarray*}
% \nonumber to remove numbering (before each equation)
  x_{n,N}^*(\mu) &=&  U_{n}^{-1} \circ \int_\RR U_{n+1}\Big( \max\big\{ x, x_{n+1,N}^*(\Phi(x,\mu))\big\}\Big)  Q(dx|\mu) \\
   &\ge & U_{n}^{-1} \circ \int_\RR U_{n+1}\Big( x_{n+1,N}^*(\Phi(x,\mu))\Big)  Q(dx|\mu) \\
   &=&   nc + U^{-1} \circ \int_\RR U\Big(  x_{n+1,N}^*(\Phi(x,\mu))-(n+1)c\Big)Q(dx|\mu) \big).
%   &=& nc + U^{-1} \Big( U\big( \int_\RR x_{n+1,N}^*(\Phi(x,\mu))Q(dx|\mu) -(n+1)c\big)\Big)\\
%   & =&\int_\RR x_{n+1,N}^*(\Phi(x,\mu))Q(dx|\mu)-c.
\end{eqnarray*}
Using the definition of the certainty equivalent yields the first statement. For the second statement note that $\rho$ is translation invariant in the case of exponential utility.
\end{proof}

In order to obtain decreasing reservation levels further assumptions are necessary. For example the property that the utility function has a {\em decreasing absolute risk aversion} (DARA). This means that $x\mapsto l_U(x)$ is decreasing, i.e. the decision maker becomes more risk averse with decreasing wealth.

\begin{theorem}\label{BRtheo:n}
\begin{itemize}
  \item[a)] The reservation levels $x_{n,N}^*(\mu)$ are increasing in $N$ for all $\mu\in\PP(E_Y)$.
  \item[b)] If the utility function $U$ is DARA, then $x_{n,N}^*(\mu)$ are decreasing in $n$ for all $\mu\in\PP(E_Y)$.
\end{itemize}
\end{theorem}

\begin{proof}
\begin{itemize}
  \item[a)] It is obvious that  $x_{N-1,N}^*$ is increasing in $N$. Now suppose that $ x_{n+1,N}^*$ is increasing in $N$. Then due to the recursion of the reservation levels in Theorem \ref{BRtheo:xrec} and the fact that $U$ is increasing we obtain that $x_{n,N}^*$ is increasing in $N$.
  \item[b)] Now suppose the utility function $U$ is DARA. Thus, in particular $U_1(x)=U(x-c)$ is more risk averse than $U$.  Then we obtain with Theorem \ref{BRtheo:risk} that
      $$ x_{n,N}^*(\mu,U) \ge x_{n,N}^*(\mu,U_1) = x_{n+1,N}^*(\mu)$$
\end{itemize}
which implies the result.
\end{proof}

\begin{example}
In case $U(x) = \frac1\gamma e^{\gamma x}$ with $\gamma <0$ we obtain from Theorem \ref{BRtheo:n} that the reservation levels $x_n^*(\mu)$ are increasing in $n$. This effect can also be seen in the numerical example \ref{BRex:num}.
\end{example}

\section{Risk-sensitive Stopping Problems with Infinite Time Horizon}\label{BRsec:ITH}
Let us now consider the risk-sensitive stopping problem from Section \ref{BRsec:stop} with infinite time horizon. Here we assume that the stopping reward $g$ is bounded, i.e. $\underline{g}\le g\le \bar{g}$ and cost are strictly negative, i.e. $\sup_{x\in E_X} c(x) =: \bar{c} < 0$. Thus we consider
\begin{equation}\label{BReq:Vinfty}
    J_\infty(x) := \sup_{\tau< \infty} \mathbb{E}_{x}\Big[ U\Big( \sum_{k=0}^{\tau-1} c(X_k)+g(X_\tau)\Big)\Big]
\end{equation}
where the supremum is taken over all $(\mathcal{F}_n)$-stopping times $\tau$ with $\PP_x(\tau<\infty)= 1$ for all $x\in E_X$.
This problem can be seen as the limiting problem of stopping problems with bounded horizon.
%Let us consider for $x\in [\underline{m},\bar{M}], \mu\in\PP(\Theta), s\in [-\frac{c}{1-\beta},0]$ and $z\in (0,1]$
%\begin{eqnarray*}
%    V_{n\tau}(x,\mu,s,z) &:=&  \int_\Theta \mathbb{E}_{x\theta}\Big[ U\Big( \beta^{\tau\wedge n} X_{\tau\wedge n} - zc\sum_{k=0}^{\tau\wedge n-1} \beta^k+s\Big)\Big] Q_0(d\theta)\\
%    V_n(x,\mu,s,z) &:=&   \sup_{\tau<\infty}  V_{n\tau}(x,\mu,s,z).
%\end{eqnarray*}
%Then we obtain

\begin{theorem}\label{BRtheo:infHor}
The sequence $(V_n)$ of value functions defined in \eqref{BReq:V-def} has a limit $V$ and this limit satisfies
  \begin{equation}\label{BReq:VFP}V(x,\mu,s) = \max\Big\{ U\big(g(x)+s\big), \int_{E_X} V\big(x',\Phi(x,x',\mu), s+c(x)\big) Q(dx'|\mu)\Big\}.\end{equation}
%\end{itemize}
\end{theorem}

\begin{proof}
We have $V_n\le V_{n+1}$ and the sequence $V_n$ is bounded from above by our assumptions, thus the limit $V=\lim_{n\to\infty} V_n$ exists. Moreover, we can take the limit on both sides in the recursion \eqref{BReq:vn} for $V_n$ to obtain the fixed point property of $V$.
\end{proof}

Next we show the relation between $J_\infty$ and $V$.

\begin{theorem}\label{BRtheo:infHor2}
\begin{itemize}
  \item[a)] It holds that $V(x,Q_0,0) = J_\infty(x)$.
   \item[b)] Let $f^*(x,\mu,s) =1$  if the maximum in \eqref{BReq:VFP} is attained at $U\big(g(x)+s\big)$ and define $(g_0^*,g_1^*\ldots)$ by
 $$ g_n^*(h_n) := f^*\Big( x_n, \mu_n(\cdot|h_n), \sum_{k=0}^{n-1} c(x_k)\Big),\quad n\in \NN_0.$$
 Then the optimal stopping time $\tau^*$ is given by
 $$ \tau^* := \inf\{ n\in\NN_0 : g_n^*(h_n) = 1\}.$$
\end{itemize}
\end{theorem}

\begin{proof}
\begin{itemize}
  \item[a)] First note that we have $V_n(x,Q_0,0) \le J_\infty(x)$ for all $n\in\NN$ which implies that $V(x,Q_0,0) \le J_\infty(x)$. Next for every admissible stopping time $\tau$ with $\PP_x(\tau<\infty)=1$ and $\mathbb{E}_{x}\Big[ |U\Big( \sum_{k=0}^{\tau-1} c(X_k)+g(X_\tau)\Big)|\Big] <\infty$  it holds for any $n\in\NN$ that
 \begin{eqnarray*} &&V(x,Q_0,0) \ge V_n(x,Q_0,0)\ge \\
 &\ge&\mathbb{E}_{x}\Big[ U\Big( \sum_{k=0}^{(\tau\wedge n)-1} c(X_k)+g(X_{\tau\wedge n})\Big)\Big]\\
 &\ge & \mathbb{E}_{x}\Big[ U\Big( \sum_{k=0}^{\tau -1} c(X_k)+g(X_{\tau\wedge n})\Big)\Big]\\
 &=&   \mathbb{E}_{x}\Big[ U\Big( \sum_{k=0}^{\tau -1} c(X_k)+g(X_{\tau})\Big)1_{[\tau \le n]}\Big]  + \mathbb{E}_{x}\Big[ U\Big( \sum_{k=0}^{\tau -1} c(X_k)+g(X_{n})\Big)1_{[\tau > n]}\Big]
 \end{eqnarray*}
  Then letting $n\to \infty$ yields with dominated convergence and the fact that $\PP_x(\tau<\infty)=1$
 $$ V(x,Q_0,0) \ge \mathbb{E}_{x}\Big[ U\Big( \sum_{k=0}^{\tau -1} c(X_k)+g(X_{\tau})\Big)\Big].$$
 Taking the supremum over all admissible stopping times and combining the result with the first inequality implies the result.
  \item[b)] Iterating the fixed point equation $n$-times and using the definition of $\tau^*$ we obtain
  \begin{eqnarray}\nonumber V(x,Q_0,0) &=& \mathbb{E}_{x}\Big[ U\Big(\sum_{k=0}^{\tau^*-1} c(X_k)+g(X_{\tau^*}) \Big) 1_{[\tau^*\le n]}\Big] +\\
  \label{BReq:xxx}&+& \mathbb{E}_{x}\Big[ V\Big(X_n,\mu_n(\cdot|H_n), \sum_{k=0}^{n-1} c(X_k)\Big)  1_{[\tau^*> n]}\Big].\end{eqnarray}
 First we show that $ \PP_x(\tau^*<\infty) =1$ for all $x\in E_X$. We can extend part a) easily to arbitrary states $(x,\mu,s)$, i.e.
 $$ V(x,\mu,s) = \sup_{\tau< \infty} \int_{E_Y} \mathbb{E}_{xy}\Big[ U\Big( \sum_{k=0}^{\tau-1} c(X_k)+g(X_\tau)+s\Big)\Big] \mu(dy).$$
 Then we obtain
 $$ V\Big(x_n,\mu_n(\cdot|h_n), \sum_{k=0}^{n-1} c(x_k)\Big)  \le U\big( n\bar{c} + \bar{g}\big)$$
 and from \eqref{BReq:xxx}
 \begin{eqnarray*}
 % \nonumber to remove numbering (before each equation)
   U\big(g(x)\big) &\le& V(x,Q_0,0)\le  \\
    &\le& U(\bar{g}) \PP_x(\tau^*\le n) + U(n\bar{c}+\bar{g}) \PP_x(\tau^* > n),
 \end{eqnarray*}
 i.e.
 \begin{eqnarray*}
 % \nonumber to remove numbering (before each equation)
   U\big(g(x)\big) &\le&  U(\bar{g}) + \Big(  U(n\bar{c}+\bar{g})-U(\bar{g})\Big) \PP_x(\tau^* > n)  \\
    &=& U(\bar{g}) + a_n \PP_x(\tau^* > n)
 \end{eqnarray*}
 where $a_n := U(n\bar{c}+\bar{g})-U(\bar{g})$.
 It holds that $a_n < 0$ for all $n\in\NN$ large enough and $\lim_{n\to\infty} a_n= -\infty$.
Hence in total
 $$ \PP_x(\tau^*>n) \le \frac{U\big(g(x)\big)-U(\bar{g})}{a_n}$$
for all $n\in\NN$ large enough which implies $\PP_x(\tau^*<\infty)=1$.
Since $V$ is bounded from above, letting $n\to\infty$ in \eqref{BReq:xxx} implies that
  $$ J_\infty(x)=V(x,Q_0,0)  \le  \mathbb{E}_{x}\Big[ U\Big(\sum_{k=0}^{\tau^*-1} c(X_k)+g(X_{\tau^*}) \Big) \Big]$$
  and $\tau^*$ is optimal.
\end{itemize}
\end{proof}

\subsection{Risk-Sensitive Bayesian House Selling Problem}
Let us consider the risk-sensitive Bayesian house selling problem with infinite time horizon.When we assume that the offers are bounded, i.e. $X_i\in[{m},{M}]$ our general assumptions of this section are satisfied.
As in the finite horizon case we can see from the fixed point equation that the optimal stopping time is characterized by {\em reservation levels.} Indeed Theorem \ref{BRtheo:infHor} and Theorem \ref{BRtheo:infHor2} apply directly and the fixed point equation reads
\begin{equation*}
    V(x,\mu,s) = \max\Big\{ U(x+s), \int_{E_X} V\big(x',\Phi(x',\mu), s-c\big) Q(dx'|\mu)\Big\}.
\end{equation*}
Then we obtain
\begin{eqnarray*}
    g_n^*(h_n) =1 \quad \Leftrightarrow \quad x_n &\ge& U_n^{-1}\Big( \int V\big(x',\Phi(x',\mu), -c(n+1)\big) Q(dx'|\mu)\Big)\\
     & =: & x_{n,\infty}^*(\mu).
\end{eqnarray*}
Note that the reservation levels in general still depend on the time stage, in contrast to the problem in the risk neutral setting because we have to memorize the cost which has accumulated so far. As for the case of finite time horizon we obtain a similar recursion for the reservation levels.

\begin{theorem}\label{BRtheo:xrecih}
\begin{itemize}
  \item[a)] The optimal stopping time for the Bayesian house selling problem with infinite time horizon is given by
$$\tau^* = \inf\big\{ n\in\NN_0 : X_n \ge x_{n,\infty}^*(\mu_n(\cdot|h_n))\big\}.$$
  \item[b)] It holds that $x_{n,\infty}^*(\mu) = \lim_{N\to\infty} x_{n,N}^*(\mu)$ and the reservation levels satisfy the following recursion
  \begin{eqnarray*}
  % \nonumber to remove numbering (before each equation)
     x_{n,\infty}^*(\mu) &=&  U_{n}^{-1} \circ \int_\RR U_{n+1}\Big( \max\big\{ x, x_{n+1,\infty}^*(\Phi(x,\mu))\big\}\Big)  Q(dx|\mu).
  \end{eqnarray*}
\end{itemize}
\end{theorem}

\begin{example}
In case $U(x) = \frac1\gamma e^{\gamma x}$ with $\gamma <0$ the recursion for the reservation levels simplifies to a fixed point equation
\begin{eqnarray*}
  % \nonumber to remove numbering (before each equation)
    x_{\infty}^*(\mu) &=&  -c+ \frac1\gamma \ln\Big( \int_\RR e^{\gamma \max\big\{ x, x_{\infty}^*(\Phi(x,\mu))\big\}}  Q(dx|\mu)\Big).
  \end{eqnarray*}
  Moreover, it holds $x_\infty^*(\mu)= \lim_{n\to\infty}x_n^*(\mu)$. From Example \ref{BRex:incinga} it follows that the reservation levels $x_{\infty}^*(\mu)$ are increasing in $\gamma$.
\end{example}

Also it is possible to discuss the influence of the attitude towards risk as before. Here we use the notation $x_{n,\infty}^*(\mu,U)$ when the reservation level belongs to utility function $U$.

\begin{theorem}
\begin{itemize}
  \item[a)] Suppose that utility function $U$ is more risk averse than utility function $W$. Then for all $n\in\NN$ we obtain that $x_{n,\infty}^*(\mu,U)\le x_{n,\infty}^*(\mu,W).$
  \item[b)] If the utility function $U$ is DARA, then $x_{n,\infty}^*(\mu)$ is decreasing in $n$.
\end{itemize}
\end{theorem}

\begin{proof}
For part a) we proceed as in the proof of Theorem \ref{BRtheo:risk} and show first that $V_n(x,\mu,s,z,U) \le r\circ V_n(x,\mu,s,z,W)$ for $n\in\NN$ where $r$ is such that $U=r\circ W$. Taking the limit yields $V(x,\mu,s,z,U) \le r\circ V(x,\mu,s,z,W)$ and the statement follows as in the proof of Theorem \ref{BRtheo:risk}. Part b) follows from part a).
\end{proof}

\section{Conclusion} \label{BRs10}
We have seen that the theory for partially observable risk-sensitive stopping problems is only slightly more complicated than the theory for the risk neutral case. We were also able to show in general that the more risk averse a decision maker is, the later she will stop. Though the numerical algorithms are more demanding than in the risk neutral case, the setting with exponential utility is still feasible.


\begin{thebibliography}{99}
%\bibitem{Bertsekas:1978} Bertsekas, D. and Shreve, S. {\it Stochastic Optimal Control.} Academic Press, NY, 1978.
%\bibitem{Feinberg:2012} Feinberg, E.: Reduction of discounted continuous-time MDPs with unbounded jump and reward rates to discrete-time total-reward MDPs. In {\it Optimization, Control, and Applications of Stochastic Systems} (D.Hernandez-Hernandez and J.A.Minjares-Sosa ed.), Birkhauser, 2012, 77-97.
%\bibitem{Zhang:2013} Zhang, Y.: Convex analytic approach to constrained discounted Markov decision processes with non-constant discount factors. {\it TOP}, {\bfseries 21} (2013) 378-408.

\bibitem{bm06}
B\"auerle, N. and M\"uller, A.: Stochastic orders and risk measures: consistency and bounds. {\it Insurance: Mathematics and Economics}, {\bfseries  38} (2006) 132-148.
\bibitem{br15}
B\"auerle, N. and Rieder, U. {\it Partially Observable Risk-Sensitive Markov Decision Processes}. {Preprint} (2015).
\bibitem{br14}
B\"auerle, N. and Rieder, U.: More risk-sensitive Markov decision processes. {\it Mathematics of Operations Research}, {\bfseries  39} (1) (2014) 105-120.
\bibitem{br11}
B\"auerle, N. and Rieder, U. {\it Markov Decision Processes with Applications to Finance.} Springer-Verlag, Berlin Heidelberg, 2011.
\bibitem{bpli03}
Bielecki, T. and Pliska, S.: Economic properties of the risk sensitive criterion for portfolio management. {\it Review of Accounting and Finance}, {\bfseries 2} (2003) 3-17.
\bibitem{cchh05}
Cavazos-Cadena, R. and Hern\'{a}ndez-Hern\'{a}ndez, D.: Successive approximations in partially observable controlled Markov chains with risk-sensitive average criterion. {\it Stochastics}, {\bfseries 77} (2005) 537-568.
\bibitem{cchh15}
Cavazos-Cadena, R. and Hern\'{a}ndez-Hern\'{a}ndez, D.: A Characterization of the Optimal Certainty Equivalent
of the Average Cost via the Arrow-Pratt Sensitivity Function. {\it Mathematics of Operations Research} (to appear), 2015.
\bibitem{crs71}
Chow, Y.S., Robbins, H. and Siegmund, D. {\it Great Expectations: The Theory of optimal Stopping.} houghton Mifflin, Boston, 1971.
\bibitem{da05}
Dana, R.A..: A representation result for concave Schur concave functions. {\it Mathematical Finance}, {\bfseries  15} (2005) 613-634.
\bibitem{dl14}
Davis, M.A.H. and Lleo, S. {\it Risk-Sensitive Investment Management.} World Scientific, 2014.
\bibitem{dims99}
Di Masi, G., and Stettner, L.: Risk sensitive control of discrete time partially observed Markov processes with infinite horizon. {\it Stochastics} {\bfseries 67}(3-4) (1999) 309-322.
\bibitem{hardy}
Hardy, G.H., Littlewood, J.E. and P\'{o}lya, G. {\it Inequalities}. Cambridge University Press, Cambridge, 1934.
\bibitem{hm72}
Howard R.A. and Matheson J.E.: Risk-sensitive {Markov Decision Processes}. {\it Management Science}, {\bfseries 18} (1972) 356--369.
\bibitem{jjr94}
James M.R., Baras, J.S. and Elliott, R.J.: Risk-sensitive control and dynamic games for partially observed discrete-time nonlinear systems.  {\it IEEE Transactions on Automatic Control}, {\bfseries 39} (1994) 780-792.
\bibitem{KMY96}
Kadota, Y.,  Kurano M. and Yasuda M.: Utility-optimal stopping in a denumerable Markov chain. {\it Bulletin of Informatics and Cybernetics}, {\bfseries 28} (1996) 15-21.
\bibitem{m07}
M\"uller, A.: Certainty equivalents as risk measures. {\it Brazilian Journal of Probability and Statistics}, {\bfseries 21} (2007) 1-12.
\bibitem{m00}
M\"uller, A.: Expected utility maximization of optimal stopping problems. {\it European Journal of Operational Research}, {\bfseries 122} (2000) 101-114.
\bibitem{shi08}
Shiryaev, A.N. {\it Optimal Stopping Rules.} Springer-Verlag, Berlin Heidelberg, 2008.
\bibitem{s04}
Stettner, L.: Risk sensitive portfolio optmization with completely and partially observed factors. {\it IEEE Transactions on Automatic Control}, {\bfseries 49} (2004) 457-464.
\bibitem{whi87}
White, D.J.: Utility, probabilistic constraints, mean and varaince of discounted rewards in Markov processes. {\it OR Spektrum}, {\bfseries 9} (1987) 13-22.
\bibitem{w90}
Whittle, P. {\it Risk Sensitive Optimal Control}. Wiley, Chichester, 1990.
\bibitem{t84}
Tamaki, M.: Optimal selection from a gamma distribution with unknown parameter. {\it Z. Oper. Res. Ser. A-B}, {\bfseries 28} (1984) 47-57.
\end{thebibliography}
\end{document}